\documentclass{article}

\usepackage[margin=1.5in]{geometry}
\usepackage[utf8]{inputenc} 
\usepackage[T1]{fontenc}    
\usepackage{hyperref}       
\usepackage{url}            
\usepackage{booktabs}       
\usepackage{amsfonts}       
\usepackage{nicefrac}       
\usepackage{microtype}      
\usepackage{xcolor, graphicx}         
\usepackage{amsthm, amsmath, amssymb}

\newtheorem{theorem}{Theorem}
\newtheorem{assumption}{Assumption}
\newtheorem{lemma}{Lemma}
\newtheorem{remark}{Remark}

\newcommand{\cl}{\mathcal}
\newcommand{\bb}{\mathbb}
\newcommand{\dr}{{\rm d}}
\newcommand{\bo}{\mathbf}

\title{On $L_2$-consistency of nearest neighbor matching}

\author{%
  James Sharpnack\thanks{Work done prior to joining Amazon.} \\
  Amazon AWS\\
  \texttt{jsharpna@gmail.com} \\
}
\begin{document}

\maketitle

\begin{abstract}
Biased sampling and missing data complicates statistical problems ranging from causal inference to reinforcement learning.  We often correct for biased sampling of summary statistics with matching methods and importance weighting.  In this paper, we study nearest neighbor matching (NNM), which makes estimates of population quantities from biased samples by substituting unobserved variables with their nearest neighbors in the biased sample.  We show that NNM is $L_2$-consistent in the absence of smoothness and boundedness assumptions in finite dimensions.  We discuss applications of NNM, outline the barriers to generalizing this work to separable metric spaces, and compare this result to inverse probability weighting.
\end{abstract}

\section{Introduction}

Issues of representation in sampling have plagued data analysis ever since the first survey was taken.
Biased sampling can be cast as a missing data problem, where data from the population of interest are partially missing, and data from the biased sample are non-missing.
Let $Y \in \bb R$ be the missing at random variable of interest, then we would like to estimate its population mean.
Suppose that we have covariates $X \in \bb R^p$ and we observe one iid sample, $\cl X_M$, from the missing population and a sample of $X,Y$ pairs, from the non-missing biased distribution (let $\cl X_N$ be the corresponding sample of $X$s).
When we know the missing density, $X \sim \mu$, and non-missing density, $X \sim \nu$ (to remember this think {$\mu$}issing, and $\nu$on-missing), we can estimate population level statistics with inverse probability weighting (IPW) and trimmed variants \cite{horvitz1952generalization, ma2020robust}.
When $\mu,\nu$ are unknown, our setting, we typically use matching methods or we estimate the density ratio and plug it into IPW.
These methods construct a weight $W_j$ for each element $X_j \in \{X_1, \ldots, X_n\} = \cl X_N$ that is dependent on $\cl X_M, \cl X_N$, and form an estimate $\hat G = \sum_{j=1}^n W_j Y_j$.
In this work, we study one of the simplest matching methods, nearest neighbor matching (NNM), which sets $W_j$ to be the proportion of elements in $\cl X_M$ for which $X_j$ is its nearest neighbor within $\cl X_N$.
We will show that it is consistent under minimal conditions.

{\bf Contributions.} Our theory is broken down into three settings of increasing difficulty and realisticness:
(1) Known $\mu$, noiseless $Y$, the only source of randomness is in the non-missing sample $\cl X_N$;
(2) Unknown $\mu$, noiseless $Y$;
(3) Unknown $\mu$, noisy $Y$, this is the standard setting of NNM.
One key lemma to establish (3) is related to the $\mu$ measure of Voronoi cells from our biased sample, and we provide a more complete characterization which may be of separate interest.
Further, we discuss the conditions and contrast NNM to IPW, highlight the barriers to generalize the theory beyond finite dimensional spaces, and discuss applications to missing data problems.

\subsection{Main Theoretical Results}

If the missingness is completely at random (MCAR) then this means that there is no dependence between the missingness and the variables $X,Y$.
In this case, nothing special needs to be done in order to consistently estimate the mean of the missing $Y$, we can just compute their empirical counterparts on the non-missing data.
Instead, we will assume we are in the more realistic situation, that our data is {\em missing at random} (MAR), which means that the missingness is independent of $Y$ conditional on the covariate $X$.

We observe iid covariates $\cl X_M \subset \bb R^p$ from density $\mu$, and $\cl X_N = \{ X_1, \ldots, X_n \} \subset \bb R^p$ from density $\nu$.
Throughout we will assume that the $X$ variables are continuous and $p$-dimensional.
We will discuss later the numerous mathematical challenges of generalizing beyond this setting.
For the non-missing data, we observe $Y_1 ,\ldots, Y_n$ from the common distribution of $Y_j \sim f_{Y|X}(.|X_j)$ (due to the MAR assumption).
Our goal is to estimate,
\[
G := \bb E[Y | \textrm{missing}] = \int y \cdot f_{Y|X}(y | x) \mu(x) \dr x.
\]

NNM can be expressed using Voronoi cells.
A Voronoi cell contains the points in $\bb R^p$ such that $X_j$ is the closest member of $\cl X_N$,
\[
S_j := \left\{ x \in \bb R^p : \| X_j - x \| = \min_{k \in [n]} \| X_k - x \| \right\},
\]
where $[n] := \{1,\ldots,n\}$ (we will ignore ties because they are measure $0$).
Let $\hat M(S) = |\cl X_M \cap S| / |\cl X_M|$ be the proportion of $\cl X_M$ within $S$.
NNM estimates $G$ by $\hat G = \sum_{j=1}^n W_j Y_j$ where $W_j = \hat M(S_j)$.
Throughout, define the measures $M(S) = \int_S \mu(x) \dr x$ and $N(S) = \int_S \nu(x) \dr x$.

{\bf Known $\mu$, noiseless $Y$}.  Throughout we will let $\eta(X) = \bb E[Y | X]$, the true regression function.
We require that $\eta$ is integrable, and will place moment conditions on it later.
In the noiseless setting $Y = \eta(X)$ almost surely (AS).
If we know $\mu$ then we do not need to rely on a finite sample $\cl X_M$, instead we use the weight $W_j = M(S_j)$.
In this special case, we will call the NNM estimator the {\em 1NN measure} of $\eta$ and denote it,
\begin{equation}
\label{eq:Q}
\bb Q_1 (\eta) := \sum_{j=1}^n M(S_j) \cdot \eta(X_j) = \int \eta(X_{(1)}(x)) \, \mu(x) \, \dr x,
\end{equation}
where $X_{(1)}(x)$ is the nearest neighbor of $x$ within $\cl X_N$.
One should think of $\bb Q_1$ as a biased sampling analogue to the empirical measure (typically denoted $\bb P_n$).

Recall that the Renyi divergence for $q_0 > 1$ is
\[
D_{q_0}(\mu || \nu) = \frac{1}{q_0 - 1} \ln \int \left(\frac{\mu(x)}{\nu(x)} \right)^{q_0} \nu(x) \dr x.
\]
Furthermore, we can take $q_0 \rightarrow 1$ to obtain the KL-divergence ($D_1$).
Throughout we will assume that $q_0,q_1 \ge 1$ are H\"older conjugates ($q_0^{-1} + q_1^{-1} = 1$).

\begin{assumption}
\label{as:measure_1}
Let $q_0 \ge 1$ be a constant then assume that the Renyi divergence is finite,
\[
D_{q_0}(\mu || \nu) < \infty.
\]
\end{assumption}

\begin{assumption}
\label{as:measure_2}
Let $q_1 \ge 1$.
The test function $\eta$ is measurable and has finite $2 q_1$ moment,
\[
\int |\eta(x)|^{2 q_1} \nu(x) \dr x < \infty.
\]
\end{assumption}

Notice that a bounded $D_{q_0}(\mu || \nu)$ is much less restrictive then assuming that the density ratio is bounded uniformly.  Furthermore, we make no smoothness assumptions on $\eta$.

\begin{theorem}
\label{thm:1NN_SLLN}
Under Assumptions \ref{as:measure_1}, \ref{as:measure_2}, the 1NN measure is consistent for $\mu$, namely,
\[
\bb Q_1 (\eta) \rightarrow \int \eta(x) \cdot \mu(x) \dr x,
\]
in $L^2$ norm as $n \rightarrow \infty$.
\end{theorem}

{\bf Unknown $\mu$, noiseless $Y$.} The main difference between this setting and the previous one is that the NNM weights require an estimate $\hat M(S_j)$ of $M(S_j)$.
The following is a relatively simple corollary to Theorem \ref{thm:1NN_SLLN}, but we state it here because it highlights the additional assumptions in this setting.

\begin{theorem}
\label{thm:noiseless_NNM}
Under Assumptions \ref{as:measure_1}, \ref{as:measure_2}, the NNM estimate is consistent for noiseless $Y$, namely,
\[
\hat G = \sum_{j=1}^n \hat M(S_j) \cdot \eta(X_j) \rightarrow \int \eta(x) \cdot \mu(x) \dr x,
\]
in $L^2$ norm as $n,m \rightarrow \infty$.
\end{theorem}

{\bf Unknown $\mu$, noisy $Y$.}
In this problem, we assume that $Y_j = \eta(X_j) + \epsilon_j$ where $\{\epsilon_j\}_{j=1}^n$ are independent with mean $0$ and variance bounded by $V < \infty$.
This result hinges on our characterization of the Voronoi cells, and requires that the Chi-square divergence is finite.

\begin{theorem}
\label{thm:noisy_NNM}
Under Assumptions \ref{as:measure_1}, \ref{as:measure_2} for $q_0 = q_1 = 2$, 
\[
\hat G = \sum_{j=1}^n \hat M(S_j) \cdot Y_j \rightarrow G,
\]
in $L^2$ norm as $n,m \rightarrow \infty$.
\end{theorem}

In the following section, we will contrast and relate these results to prior work.

\subsection{Comparison to prior work}

Matching methods for causal inference and missing data are appealing due to their relative 
simplicity, interpretability, and computational efficiency \cite{heckman1998matching}.
Matching can be done with replacement, where multiple missing samples can match to the same non-missing samples, such as NNM, or without replacement, such as optimal matching \cite{rosenbaum1989optimal}; it has been shown that in some circumstances matching without replacement is inconsistent \cite{savje2019inconsistency}.
Increasingly, NNM has surfaced in machine learning applications such as covariate shift correction in classification \cite{loog2012nearest}, model based conditional independence testing \cite{sen2017model}, and deep clustering \cite{dang2021nearest}.
NNM is amenable to massive data applications due to fast approximate NN indexing and retrieval (for example, \cite{muja2014scalable,malkov2018efficient}), and software for NNM and extensions have been developed \cite{abadie2004nnmatch}.


The statistical efficiency of NNM has been studied in statistics and econometrics literature.
In \cite{abadie2006large}, it was shown that despite its popularity, in more than 1 dimensions NNM (and other similar matching methods) has a bias term that converges at a rate of $n^{-1/p}$, markedly worse than the optimal $n^{-1/2}$ rate under Lipschitz continuity assumptions and bounded propensity score (the conditional probability of missingness).
For this reason, corrective measures have been studied such as bias-correction, \cite{abadie2011bias}, where an overparametrized linear model is used to construct an additive bias estimate.
Another work in response to the negative results of \cite{abadie2006large} form an estimate of the propensity score and use NNM in this 1 dimensional space \cite{frolich2004finite, busso2014new}, however, this requires an accurate estimate of the propensity score.
To the best of our knowledge, it remained unknown if NNM is even consistent without smoothness assumptions or bounded density ratio (the results in this work).


As we will see, Theorem \ref{thm:1NN_SLLN} relies on existing results on nearest neighbor regression theory.
To illustrate this under stronger assumptions---$\eta$ is $L$-Lipschitz and bounded---then $|\eta(X_{(1)}(x)) - \eta(x)| \le L \| X_{(1)}(x) - x \|$. 
We know by classical studies of KNN, \cite{cover1967nearest}, that the 1NN approaches the test point and boundedness gives us dominated convergence.
Since, this work much has been discovered about KNN that we can potentially stand on the shoulders of.
Of direct relevance to this work is \cite{forzani2012consistent}, which studies the $L_2$ consistency of KNN regression in general metric spaces.
Also, intermediate results from \cite{gyorfi2021universal, hanneke2021universal} are relevant as well, particularly the density of Lipschitz functions in $L_1(\nu)$ for separable metric spaces.
We will argue in Section \ref{sec:closer} that, while promising, these results are insufficient to prove our desired Theorems.

We will see that to prove the final result, Theorem \ref{thm:noisy_NNM}, we will require a characterization of the Voronoi cells.
\cite{devroye2017measure} provides an analogous result for unbiased sampling, and we extend this result to the biased sampling setting.
They find that $\bb E[M(S_j) | X_j = x] \rightarrow 1$ when $\mu = \nu$ and bound the limiting second moment.
We extend this to find that $\bb E[M(S_j) | X_j = x] \rightarrow \mu(x) / \nu(x)$, which implies that NNM is unbiased for IPW.  As mentioned, by \cite{abadie2006large} we know that this bias converges at a suboptimal rate under smoothness and boundedness assumptions.  It is unknown, and outside of the scope of this work, if without these assumptions the optimal rate remains $n^{-1/2}$ or if this is information theoretically impossible.

\section{Asymptotic measure of Voronoi cells}

The main observation that we make in this section is that
the $\mu$-measure of a Voronoi cell of samples from $\nu$ approaches the density ratio.
These results rely heavily on the finite dimensional setting, and it is used to prove NNM consistency in the unknown $\mu$, noisy $Y$ case (Theorem \ref{thm:noisy_NNM}).
In fact, that result only requires Lemma \ref{lem:voronoi} \eqref{eq:voronoi_sqr}, but we provide our full result here because it is enlightening.
This result can be interpreted as NNM is unbiased for importance sampling in the limit, since $M(S_j)$ is the expected importance weight.
To make this leap, we need to be specific regarding our Lebesgue points, which is a $\nu$ measure 1 set over which this limit holds.  We follow this with our Voronoi cell result.

\begin{lemma}
\label{lem:lebesgue_pts}
For any probability measure $\Pi$ with a density over $\bb R^p$.  There exists a set $\cl X$ such that $\Pi(\cl X) = 1$ and the following properties hold:
\begin{enumerate}
\item[(1)] For any $x \in \cl X$ and $x' \ne x$, $\Pi(B(x',\|x' - x \|)) > 0$ and
\item[(2)] for any $x \in \cl X$ and any $\delta > 0$ there exists an $\zeta > 0$ depending on $x$ such that if $\Pi(B(x', \|x - x'\|)) \le \zeta$ for some $x' \ne x$ then we have that $\|x - x'\| < \delta$.
\end{enumerate}
\end{lemma}


\begin{lemma}
\label{lem:voronoi}
Under Assumption \ref{as:measure_1}, in expectation, the $M$-measure of a Voronoi cell around $X_1$ conditional on $X_1$ converges to the density ratio in the limit, namely,
\begin{equation}
\label{eq:voronoi}
\lim_{n \rightarrow \infty} n \bb E [M(S_1) | X_1 = x] = \frac{\mu(x)}{\nu(x)},
\end{equation}
for $\nu$-almost all $x$ (where the Lebesgue points are those described in Lemma \ref{lem:lebesgue_pts} for $\Pi = N$).
Furthermore, we have the following bound,
\begin{equation}
\label{eq:voronoi_sqr}
\limsup_{n \rightarrow \infty} n^2 \bb E [M^2(S_1) | X_1 = x] \le 2 \left(\frac{\mu(x)}{\nu(x)} \right)^2.
\end{equation}
\end{lemma}

\begin{remark}
  The proof, in the appendix, borrows some tricks from the corresponding result in \cite{devroye2017measure}, although we must adapt their proof to accommodate $\mu \ne \nu$.
  There seems to be an issue with the validity of the proof of Theorem 2.1(i) in \cite{devroye2017measure}, particularly how the Lebesgue differentiation theorem applied to fixed points $v$ and $x$ can then translate to the similar result uniformly over $X_0$ (which is a draw from $\mu$).
  Our more complete study of the Lebesgue points in Lemma \ref{lem:lebesgue_pts} resolves this potential oversight, completing and generalizing the proof.
\end{remark}

\begin{proof}[Proof sketch of Lemma \ref{lem:voronoi}]
Recall that $S_1$ is the Voronoi cell around $X_1$.  As was done in \cite{devroye2017measure} (in the case that $\mu = \nu$), we observe that for $X_0 \sim \mu$,
\begin{align*}
&\bb E[M(S_1) | X_1 = x] = \bb P \{ X_0 \in S_1 | X_1 = x\} \\
& = \bb P \left\{ \cap_{i=2}^n \left\{ X_i \notin B(X_0,\| X_0 - x\|) \right\} \right\} = \bb E \left[(1 - Z(x))^{n-1} \right],
\end{align*}
where $Z(x) = N (B(X_0,\| X_0 - x\|))$.
In the appendix, we use the Lebesgue differentiation theorem (LDT) and Lemma \ref{lem:lebesgue_pts} to make precise the following string of approximations
\begin{align}
&N(B(X_0, \|X_0 - x \|)) \approx \nu(x) \lambda(B(X_0, \|X_0 - x \|)) \nonumber \\
&= \nu(x) \lambda(B(x, \|X_0 - x \|)) \approx \frac{\nu(x)}{\mu(x)} M(B(x, \|X_0 - x \|)),
\label{eq:lebesgue_approx}
\end{align}
and it is straightforward to see that $Z_0(x) = M(B(x, \|X_0 - x \|))$ has a uniform distribution,
which after some derivation gives us \eqref{eq:voronoi}.
In order to establish \eqref{eq:voronoi_sqr}, we follow a similar procedure.
\end{proof}

We can see the necessity of the assumption that these have densities with respect to the Lebesgue measure due to our use of $\lambda(B(X_0, \|X_0 - x \|)) = \lambda(B(x, \|X_0 - x \|))$ used in \eqref{eq:lebesgue_approx}.

\section{Proving $L_2$-consistency of NNM}

This section is primarily devoted to proving $L_2$-consistency in the known $\mu$, noiseless $Y$ case, Theorem \ref{thm:1NN_SLLN}.
In order to prove Theorems \ref{thm:noiseless_NNM}, \ref{thm:noisy_NNM} we control the additional randomness due to $\cl X_M$ and noisy $Y$.
Both proofs are in the appendix, so the results are not restated here.
For Theorem \ref{thm:noisy_NNM}, we require Lemma \ref{lem:voronoi}.
In short, the variance of the summand in $\hat G$, $M(S_j) Y_j$, is bounded by $V M^2(S_j)$ so we need to control the squared $\mu$-measure of the Voronoi cells.


We will divide the proof of Theorem \ref{thm:1NN_SLLN} into two thrusts: demonstrating asymptotic unbiasedness and diminishing variance.
We will discuss in Section \ref{sec:closer} how these results might be able to generalize to separable metric spaces.

{\bf Asymptotic unbiasedness} of $\bb Q_1(\eta)$ follows almost immediately from finite dimensional nearest neighbor theory \cite{biau2015lectures} and H\"older's inequality.  We give a proof sketch here to highlight how it might easily generalize to metric spaces, in the event that new NN regression theory is developed.

\begin{theorem}
\label{thm:bias}
Let $q_0, q_1$ be H\"older conjugates, suppose Assumption \ref{as:measure_1} and that $\bb E |\eta(X_1)|^{q_1} < \infty$.
Then
\[
\lim_{n \rightarrow \infty} \bb E [\bb Q_1(\eta)] = \int \eta(x) \mu(x) \dr x.
\]
\end{theorem}

\begin{proof}
By H\"olders inequality, 
\begin{align*}
&\left|E [\bb Q_1(\eta)] - \int \eta(x) \mu(x) \dr x\right| = \left| \bb E \left[ \int (\eta(X_{(1)}(x)) - \eta(x)) \mu(x) \dr x \right] \right|\\
&\le \left( \int |\eta(X_{(1)}(x)) - \eta(x)|^{q_1} \nu(x) \dr x \right)^{1/q_1} \cdot \left( \int \left( \frac{\mu(x)}{\nu(x)} \right)^{q_0} \nu(x) \dr x \right)^{1/q_0}
\end{align*}
by Lemma \ref{lem:stone2} (see \cite{biau2015lectures}) from classical NN theory, we have that
\begin{equation}
\label{eq:1nn_conv}
\int |\eta(X_{(1)}(x)) - \eta(x)|^{q_1} \nu(x) \dr x \rightarrow 0.
\end{equation}
which completes the proof (in fact it shows $L_1$ convergence).
\end{proof}

One can gain a better intuition by proving this using Lemma \ref{lem:voronoi}.
Specifically, the expected 1NN measure is,
\[
\bb E [\bb Q_1(\eta)] = \bb E \left[\eta(X_1) \cdot n \bb E[ M(S_1) | X_1 ] \right].
\]
We have pointwise convergence by \eqref{eq:voronoi},
\[
\eta(X_1) \cdot n \bb E \left[ M(S_1) | X_1 \right] \rightarrow \frac{\mu(X_1)}{\nu(X_1)} \eta(X_1),
\]
almost everywhere, and the RHS has expectation $\int \eta \mu$.
What remains is to show dominated convergence  (see the alternative proof of Theorem \ref{thm:bias} in Appendix).
We also demonstrate in the Appendix using instructive examples that for finite $n$ the bias is unavoidable.
These are typically cases where the LDT has non-uniform convergence (see \eqref{eq:lebesgue_approx}).

{\bf Diminishing variance.} We have established that the 1NN measure is asymptotically unbiased, but $L_2$-consistency remains to be shown.
Our main tool for showing this consistency is the following variance bound, which holds without any additional assumptions then those stated within.
Lemma \ref{lem:var_bd} demonstrates that as long as $\mu$ and $\nu$ are not too dissimilar, the variance of the 1NN measure is bounded by the discrepancy between the first and second nearest neighbor interpolants.

\begin{lemma}
\label{lem:var_bd}
Let $q_0,q_1$ be H\"older conjugates then,
\[
\mathbb V \left( \bb Q_1 (\eta) \right) \le 2 \left(\int \left(\frac{\mu(x)}{\nu(x)}\right)^{q_0} \nu(x) \dr x \right)^{\frac{1}{q_0}} \left(\int \left| \eta(X_{(1)}(x)) - \eta(X_{(2)}(x)) \right|^{2 q_1} \nu(x) \dr x \right)^{\frac{1}{q_1}}.
\]
\end{lemma}

The fact that $\mu, \nu$ are densities or even over $\bb R^p$ is actually not required.
If one were to replace $\mu/\nu$ with the Radon-Nikodym derivative then the result would still hold.
We conclude this subsection by showing that the 1NN measure has diminishing variance.

\begin{theorem}[1NN measure variance]
\label{thm:var}
Under Assumptions \ref{as:measure_1} and \ref{as:measure_2}, we have that
\[
\mathbb V \left( \bb Q_1(\eta) \right) \rightarrow 0, \quad \textrm{as } n \rightarrow \infty.
\]
\end{theorem}

\begin{proof}
In Lemma \ref{lem:two_NN_diff} (in Appendix) we establish that under Assumption \ref{as:measure_2} we have that for $X \sim \nu$,
\[
\bb E | \eta(X_{(1)}(X)) -  \eta(X_{(2)}(X))|^{2q_1} \rightarrow 0.
\]
This result uses lemmata from the study of nearest neighbors regressors in \cite{biau2015lectures}.
Under Assumption \ref{as:measure_1}, we have that $\bb E (\mu(X)/\nu(X))^{q_0}$ is bounded.
Applying Lemma \ref{lem:var_bd} we reach our conclusion.
\end{proof}




\section{A closer look at the results and their assumptions}
\label{sec:closer}

This section demonstrates some implications and potential generalizations of the above results.
First, we discuss Assumptions \ref{as:measure_1}, \ref{as:measure_2} and show that the 1NN measure is $L_2$-consistent in situation where IPW is not.
Second, we discuss potential generalizations to separable metric spaces and the major places in which the finite dimensional assumption is required in this work.

\subsection{Comparison to Inverse Probability Weighting (IPW)}
\label{sec:IPW}

{\bf Comparing consistency conditions.} 
For this comparison, it is sufficient to consider the known $\mu$, noiseless $Y = \eta(X)$ case.
We will see that there are situations in which NNM achieves consistency where IPW is not guaranteed consistency.
The IPW estimate can be expressed as $\bb P_n (\tilde \eta)$ where $\tilde \eta(x) = \eta(x) \cdot \mu(x) / \nu(x)$.
The $L_2$ weak law of large numbers states that if $\bb V(\tilde \eta(X)) < \infty, X \sim \nu$, i.e.~has finite second moment, then we have that $\bb P_n (\tilde \eta) \rightarrow \int \eta \mu$.
Hence, we can compare this condition, 
\begin{equation}
\label{eq:IPW_cond}
\int \left( \eta(x) \cdot \frac{\mu(x)}{\nu(x)} \right)^2 \nu(x) \dr x < \infty,
\tag{IPW Condition}
\end{equation}
to the Assumptions \ref{as:measure_1}, \ref{as:measure_2}.
To provide a natural comparison, we will use H\"olders inequality, to obtain, 
\[
D_{2 q_0}(\mu || \nu) < \infty, \quad \int |\eta(x)|^{2 q_1} \nu(x) \dr x < \infty,
\]
as a stronger IPW condition, that is tight for some examples.
Notice that this is a stronger condition than the Assumptions \ref{as:measure_1}, \ref{as:measure_2}, leaving us with the result that the 1NN measure is $L_2$-consistent in situations where $L_2$-consistency of IPW is not guaranteed.

{\bf Example where NNM is better than IPW.} We construct one such example from the Student's t-distribution.
Let $\nu$ be $t_k$-distributed with the degrees of freedom $k \in (3,4)$, $\mu$ be $t_{k-1}$ distributed, and $\eta(x) = |x|$.
The density ratio $\mu(x) / \nu(x) = C_k (1 + x^2 / k)^{-k/2} / (1 + x^2 / (k-1))^{-(k+1)/2}$.
Then \eqref{eq:IPW_cond} does not hold:
\[
\int \left( \eta(x) \cdot \frac{\mu(x)}{\nu(x)} \right)^2 \nu(x) \dr x \ge C_k \int x^2 \cdot (1 + x^2) \nu(x) \dr x= \infty,
\]
since $\nu$ does not have finite fourth moment (for a constant $C_k$).
However, we can select $q_0 = 3$ and $q_1 = 3/2$ to see that Assumptions \ref{as:measure_1}, \ref{as:measure_2} hold since,
\[
\int \left(\frac{\mu(x)}{\nu(x)} \right)^{q_0} \nu(x) \dr x \le C_k' \int (1 + x^2)^{3/2} \nu(x) \dr x < \infty,
\]
and $\int \eta^{2 q_1} \nu = \int |x|^3 \nu(x) \dr x < \infty$, both because $\nu$ has finite third moment.
We can see that in simulation this bears out and NNM has lower mean squared error than IPW (Table \ref{tab:student}).

Of course, in this example, one would use the trimmed variant of the IPW \cite{ma2020robust}, where we replace the IPW with $W_i = (\mu(X_i) / \nu(X_i)) \cdot \bo 1 \{ \mu(X_i) / \nu(X_i) < b_n \}$.
This trimming introduces bias, but as long as $\int \eta^2 \nu < \infty$ we can obtain $L_2$ consistency by letting $b_n \rightarrow \infty$ (perhaps extremely slowly).
It is worthwhile to remember that NNM does not require knowledge or an estimate of $\mu/\nu$, while IPW and its trimmed variant does.
One can interpret these observations as the following: NNM implicitly trims the importance weight, trading off more bias for less variance.

\begin{table}
\centering
\begin{tabular}{lrrrrrr}
\toprule
n        & 16 & 64 & 256 & 1024 & 4096 & 16384 \\
\midrule
NNW Mean &  0.990 &  1.155 &   1.236 &    1.284 &    1.312 &     1.330 \\
NNW MSE  &  0.149 &  0.045 &   0.016 &    0.006 &    0.002 &     0.001 \\
IPW Var  &  1.178 &  0.485 &   0.097 &    0.048 &    0.032 &     0.032 \\
\bottomrule
\end{tabular}
\caption{A comparison of NNM and IPW for the t-distribution example, where the true value is $\bb E_\mu |X| \approx 1.356$.  IPW is unbiased, so the variance equals the mean square error (MSE).}
\label{tab:student}
\end{table}

\subsection{Generalizing to separable metric spaces}

The restrictiveness of requiring $X$ to be continuous and finite dimensional is striking when we compare these results to what we know about KNN classification \cite{hanneke2021universal} and Proto-NN \cite{gyorfi2021universal}.
In this section we will highlight all of the places in which the finite dimensionality (FD) assumption is used in this paper and discuss approaches to generalizing to separable metric spaces.

{\bf Noiseless $Y$.}
For the proof of Theorem \ref{thm:1NN_SLLN}, the only real place that the FD assumption was used is \eqref{eq:1nn_conv}.
In fact, we can use a recent result from \cite{gyorfi2021universal} to establish consistency of the 1NN measure for separable metric spaces but under significantly more restrictive Assumptions than \ref{as:measure_1}, \ref{as:measure_2}.
In that work, they show (Theorem 3) that ProtoNN is pointwise $L1$-consistent for classification, and in the proof they show that when $\eta$ is bounded AS
\[
\int | \eta(X_{(1)}(x)) - \eta(x) | \nu(x) \dr x \rightarrow 0.
\]
This is exactly \eqref{eq:1nn_conv} with $q_1 = 1$ but with an additional boundedness assumption. 
If then the density ratio $\mu(x) / \nu(x)$ is also bounded AS, this implies that $\bb Q_1(\eta)$ is $L_1$-consistent (but not necessarily $L_2$-consistent).
Of course, a bounded density ratio and bounded $\eta$ dramatically weaken the result, making it not applicable to estimating expectations, variances, and many other moments, as well as not applicable to distributions such as normals, gammas, betas, etc.

It is worth attempting to weaken these assumptions and establish $L_2$ consistency using directly the proof techniques in \cite{hanneke2021universal}, but there are specific barriers.
First, one of the main tools used is the density of Lipschitz functions in $L_1(\nu)$ (where now $\nu$ is a Borel measure).
However, we would require that Lipschitz functions are dense in $L_p(\nu)$, which has not been established to the best of our knowledge (although we have no counter-example).
Furthermore, the boundedness of $\eta$ is used to establish dominated convergence, and it is unclear how to get around this.
To the best of our knowledge, establishing \eqref{eq:1nn_conv} under only moment assumptions in separable metric spaces is an open problem.
Such a result would also be able to be used to tackle Theorem \ref{thm:noiseless_NNM}---unknown $\mu$, noiseless $Y$.

Finally, the proof of Theorem 4.3 in \cite{forzani2012consistent} indicates that \eqref{eq:1nn_conv} may be established for $q_1 = 2$ for bounded functions in metric spaces that satisfy the Besicovitch condition.
Of course, the boundedness condition violates our assumptions, but the proof of the extension of Stone's theorem (Theorem 3.4) contains an infinite dimensional analogue of Lemma \ref{lem:stone2}.
However, that result relies on a somewhat opaque condition (iii') and it is unclear if it can be generalized to $L_4$-convergence, which is needed for Theorem \ref{thm:noiseless_NNM}.
In summary, there are promising approaches to generalizing the noiseless case to metric spaces, however, it is safely outside of the scope of this work.

{\bf Noisy $Y$.}
The proof of Theorem \ref{thm:noisy_NNM} required the use of Lemma \ref{lem:voronoi} \eqref{eq:voronoi_sqr}.
This was required to establish the convergence, $\bb E[\sum_j M^2(S_j)] \rightarrow 0$, and it is unclear how to do this without our characterizations of the $\mu$ measure of Voronoi cells.
This condition is unavoidable, because the conditional variance of $\bb V(\hat G | \cl X_N) = V \sum_j M^2(S_j)$ for known $\mu$ and constant $\bb V (\epsilon_i) = V$. 
As mentioned these results were heavily reliant on the FD assumptions and continuous $X$, since we appealed to the translation invariance of the Lebesgue measure.
Furthermore, the only precedent that we have of characterizing Voronoi cells is \cite{devroye2017measure} which is also in the FD setting.
As mentioned, it may be that a weaker result than Lemma \ref{lem:voronoi} would be sufficient.

\section{Applications to missing data problems}

\subsection{Imputation in massive databases}

We will consider statistics that are aggregates of non-linear elementwise operations (i.e. empirical moments).
Most common aggregations on database tables, such as sum, mean, variance, covariance, and count along with grouping operations and filters can be expressed in this way.
Specifically, let $Z \in \bb R^d$ be a partially missing random variable and $g : \bb R^p \times \bb R^d \to \bb R$ be a possibly non-linear integrable function then we will focus on estimating the following functional,
\[
G := \bb E \left[ g(X,Z) | {\rm missing} \right],
\]
which is the expectation of $g(X,Z)$ for the missing population.
For example, suppose we would like to express the following query, {\tt select mean(log(Z)) where X < 1 and Z = missing},
we could use the function $g(x,z) = 1\{x < 1\} \log z$ (in this example, $p=d=1$).
Of course, we are not able to make such a query because it is based on unobserved data.
NNM is equivalent to redefining $Y \gets g(X,Z)$ and performing single imputation on the new $Y$ with the nearest neighbor in $X$ space.
However, this can be done implicitly by precomputing the NNM weights based on $X$, and then computing $\hat G$ for any arbitrary $g$ (without the need to recompute new weights for new $g$).
The NNM weights need to be updated only when the index is modified via insert, delete, etc.
These aggregate computations can be implemented with search indexing with approximate nearest neighbor, a standard technology for indexing in distributed databases. 

\subsection{Imputation of the trans-Atlantic slave trade}

The trans-Atlantic slave trade (TAST), also known as the middle passage, refers to the slave ship voyages that brought African slaves to the Americas.
The middle passage is reported to have forcibly migrated over 10 million Africans to the Americas over a roughly 3 century time span.
The number of slaves that embarked from Africa is especially important since the number of slaves taken from Africa can impact other estimates that result from this.
For example, when estimating the population of Africa in a given decade, demographers will use population growth models and more recent census data \cite{manning2020research}.
However, the population growth was stifled by the slave trade, and without accounting for it past populations will tend to be underestimated because the growth rate is overestimated.

The database that we use is the 2010 extended version of the Voyages database, \cite{eltis2020digital}.
There is a significant amount of missingness throughout the database---$76.5\%$ of the voyages have missing number of slaves at embarkation---which is the partially missing variable of interest.
We apply NNM to compute the total number of slaves taken from Africa using the number of slave at arrival and the year for the voyage as covariates.
In Figure \ref{fig:slaves}, we can see the non-missing data and the 1NN imputed data (missing $Y$s filled in with its matched value).
The NNM estimate of the total number of slaves taken from Africa is $10$,$644$,$376$, while the MCAR assumption over-estimates this---$11$,$569$,$160$.

\begin{figure}
\centering
\includegraphics[width=.8\textwidth]{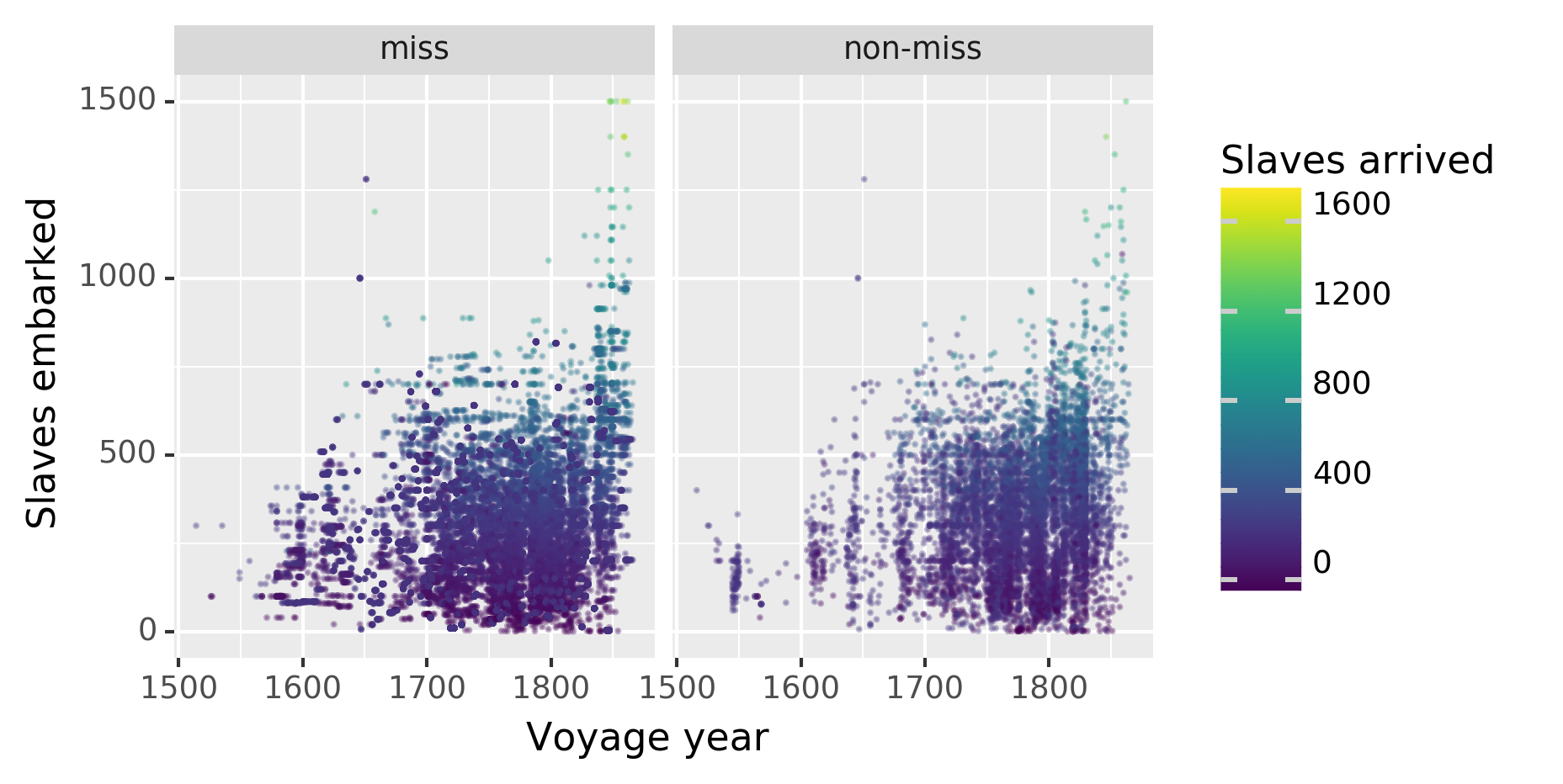}
\caption{Scatterplot of slave trade voyage's number of slaves embarked (y-axis), voyage year (x-axis), number of slaves disembarked (color), and the 1NN weight (size).  The missing and 1NN imputed number of slaves arrived in Americas are on the left and non-missing values are on the right.}
\label{fig:slaves}
\end{figure}

\subsection{Assessing test loss under covariate shift}

When the training and test datasets in supervised learning have different covariate distributions, then we have covariate shift \cite{shimodaira2000improving}.
Let $Y \in \bb R^d$, $\Omega_0$ be the training data, $\Omega_1$ the validation data, and $\Omega_2$ the test data.
By training a predictor $\hat h : \bb R^p \to \bb R^d$ on $\Omega_0$, we can consider this fixed and obtain the validation losses $L_i = \ell(\hat h(X_i), Y_i)$ for each $i \in \Omega_1$.
The test error can be estimated using NNM where $L$ is missing on the test data $\Omega_2$ and non-missing on $\Omega_1$.
Going beyond this, \cite{loog2012nearest} has used NNM to perform domain adaptation where $\hat h$ is directly trained using a test error estimate with NNM.
However, to demonstrate the validity of this approach we require uniform laws of large numbers, a future direction of research.
Similarly, finite sample rates of convergence would be required to establish generalization error bounds.
Overall, such results are a natural followup to this work.

\bibliographystyle{abbrv}
\bibliography{nnimpute}

\begin{thebibliography}{10}

\bibitem{abadie2004nnmatch}
A.~Abadie, J.~L. Herr, G.~Imbens, and D.~M. Drukker.
\newblock Nnmatch: Stata module to compute nearest-neighbor bias-corrected
  estimators, 2004.

\bibitem{abadie2006large}
A.~Abadie and G.~W. Imbens.
\newblock Large sample properties of matching estimators for average treatment
  effects.
\newblock {\em econometrica}, 74(1):235--267, 2006.

\bibitem{abadie2011bias}
A.~Abadie and G.~W. Imbens.
\newblock Bias-corrected matching estimators for average treatment effects.
\newblock {\em Journal of Business \& Economic Statistics}, 29(1):1--11, 2011.

\bibitem{biau2015lectures}
G.~Biau and L.~Devroye.
\newblock {\em Lectures on the nearest neighbor method}.
\newblock Springer, 2015.

\bibitem{busso2014new}
M.~Busso, J.~DiNardo, and J.~McCrary.
\newblock New evidence on the finite sample properties of propensity score
  reweighting and matching estimators.
\newblock {\em Review of Economics and Statistics}, 96(5):885--897, 2014.

\bibitem{cover1967nearest}
T.~Cover and P.~Hart.
\newblock Nearest neighbor pattern classification.
\newblock {\em IEEE transactions on information theory}, 13(1):21--27, 1967.

\bibitem{dang2021nearest}
Z.~Dang, C.~Deng, X.~Yang, K.~Wei, and H.~Huang.
\newblock Nearest neighbor matching for deep clustering.
\newblock In {\em Proceedings of the IEEE/CVF Conference on Computer Vision and
  Pattern Recognition}, pages 13693--13702, 2021.

\bibitem{devroye2017measure}
L.~Devroye, L.~Gy{\"o}rfi, G.~Lugosi, and H.~Walk.
\newblock On the measure of voronoi cells.
\newblock {\em Journal of Applied Probability}, 54(2):394--408, 2017.

\bibitem{eltis2020digital}
D.~Eltis.
\newblock Digital resources: The trans-atlantic slave trade database, 2020.

\bibitem{forzani2012consistent}
L.~Forzani, R.~Fraiman, and P.~Llop.
\newblock Consistent nonparametric regression for functional data under the
  stone--besicovitch conditions.
\newblock {\em IEEE transactions on information theory}, 58(11):6697--6708,
  2012.

\bibitem{frolich2004finite}
M.~Fr{\"o}lich.
\newblock Finite-sample properties of propensity-score matching and weighting
  estimators.
\newblock {\em Review of Economics and Statistics}, 86(1):77--90, 2004.

\bibitem{gyorfi2021universal}
L.~Gy{\"o}rfi and R.~Weiss.
\newblock Universal consistency and rates of convergence of multiclass
  prototype algorithms in metric spaces.
\newblock {\em Journal of Machine Learning Research}, 22(151):1--25, 2021.

\bibitem{hanneke2021universal}
S.~Hanneke, A.~Kontorovich, S.~Sabato, and R.~Weiss.
\newblock Universal bayes consistency in metric spaces.
\newblock {\em The Annals of Statistics}, 49(4):2129--2150, 2021.

\bibitem{heckman1998matching}
J.~J. Heckman, H.~Ichimura, and P.~Todd.
\newblock Matching as an econometric evaluation estimator.
\newblock {\em The review of economic studies}, 65(2):261--294, 1998.

\bibitem{horvitz1952generalization}
D.~G. Horvitz and D.~J. Thompson.
\newblock A generalization of sampling without replacement from a finite
  universe.
\newblock {\em Journal of the American statistical Association},
  47(260):663--685, 1952.

\bibitem{loog2012nearest}
M.~Loog.
\newblock Nearest neighbor-based importance weighting.
\newblock In {\em 2012 IEEE International Workshop on Machine Learning for
  Signal Processing}, pages 1--6. IEEE, 2012.

\bibitem{ma2020robust}
X.~Ma and J.~Wang.
\newblock Robust inference using inverse probability weighting.
\newblock {\em Journal of the American Statistical Association},
  115(532):1851--1860, 2020.

\bibitem{malkov2018efficient}
Y.~A. Malkov and D.~A. Yashunin.
\newblock Efficient and robust approximate nearest neighbor search using
  hierarchical navigable small world graphs.
\newblock {\em IEEE transactions on pattern analysis and machine intelligence},
  42(4):824--836, 2018.

\bibitem{manning2020research}
P.~Manning and Y.~Liu.
\newblock Research note on captive atlantic flows: Estimating missing data by
  slave-voyage routes.
\newblock {\em Journal of World-Systems Research}, 26(1):103--125, 2020.

\bibitem{muja2014scalable}
M.~Muja and D.~G. Lowe.
\newblock Scalable nearest neighbor algorithms for high dimensional data.
\newblock {\em IEEE transactions on pattern analysis and machine intelligence},
  36(11):2227--2240, 2014.

\bibitem{rosenbaum1989optimal}
P.~R. Rosenbaum.
\newblock Optimal matching for observational studies.
\newblock {\em Journal of the American Statistical Association},
  84(408):1024--1032, 1989.

\bibitem{savje2019inconsistency}
F.~S{\"a}vje.
\newblock On the inconsistency of matching without replacement.
\newblock {\em arXiv preprint arXiv:1907.07288}, 2019.

\bibitem{sen2017model}
R.~Sen, A.~T. Suresh, K.~Shanmugam, A.~G. Dimakis, and S.~Shakkottai.
\newblock Model-powered conditional independence test.
\newblock {\em Advances in neural information processing systems}, 30, 2017.

\bibitem{shimodaira2000improving}
H.~Shimodaira.
\newblock Improving predictive inference under covariate shift by weighting the
  log-likelihood function.
\newblock {\em Journal of statistical planning and inference}, 90(2):227--244,
  2000.

\bibitem{stone1977consistent}
C.~J. Stone.
\newblock Consistent nonparametric regression.
\newblock {\em The annals of statistics}, pages 595--620, 1977.

\bibitem{zajivcek1987porosity}
L.~Zaj{\'\i}{\v{c}}ek.
\newblock Porosity and $\sigma$-porosity.
\newblock {\em Real Analysis Exchange}, 13(2):314--350, 1987.

\end{thebibliography}

\newpage

\appendix

\section{Explanation and examples}

We will examine a few examples which put this theory to the test, and see numerically the convergence guaranteed in Theorem \ref{thm:1NN_SLLN}.
Our variance bound in Lemma \ref{lem:var_bd} is determined by the $L^{q_1}(\nu)$ norm of $\eta$ and $D_{2 q_0}(\mu || \nu)$.
It is instructive to go over the outline the proof of Lemma \ref{lem:voronoi}, because the proof indicates which models will yield more slowly diminishing bias than others.

\paragraph{Example 1.}

Let $\nu$ be Beta$(1.25,1)$ and $\mu$ be Beta$(.75,1)$.
Then $\mu(x) \propto x^{-0.25}$ and $\nu(x) \propto x^{0.25}$, hence the density ratio $\mu(x)/\nu(x) \propto x^{-0.5}$ is diverging as $x \rightarrow 0$.
This is an example where $\mu,\nu$ have the same compact support.
An unbounded density ratio causes challenges for the 1NN measure because it means that near $0$ there is a significant amount of mass in $\mu$ but few data from $\nu$ to evaluate $\eta$.
We assessed the measure $M(S_j)$ by Monte Carlo sampling with 1M samples from $\mu$.

Figure \ref{fig:gaussian_beta} (right) depicts the density ratio and the $M$-measure of the Voronoi cells.
Because the Voronoi cells are random, we have that the measure is only on the average approaching the density ratio, and there is significant spread around the density ratio for a given location $X_i = x$.
Let $\eta(x) = x^{-0.25}$, and we can see that $D_2(\mu || \nu) < \infty$ and $\int \eta(X_1)^4 < \infty$, satisfying the assumptions.
Despite having unbounded density ratio, $\bb Q_1 (\eta)$ converges to its limiting expectation ($1.5$) as we can see in Table \ref{tab:ex}.
We can see that the spread of $M(S_1) | X_1 = x$ is greater for the larger density ratios, and furthermore, for finite samples this is biased downward for $x$ near 0.

\begin{table}
\centering
 \begin{tabular}{|l|r|r|r|r|r|}
 \hline
  Example & $n = 1e2$ & $1e3$ & $1e4$ & $1e5$ & $\sim \mu$ \\
 \hline
 1. Beta & $1.472$ & $1.483$ & $1.492$ & $1.493$ & $1.5$ \\
 2. Gaussian & $1.774$ & $1.991$ & $2.063$ & $2.083$ & $2.1$ \\
 3. Fat Cantor & $1.587$ & $1.842$ & $1.970$ & $1.997$ & $2$ \\
 \hline
 \end{tabular}
 \caption{Monte Carlo samples of $\bb Q_1(\eta)$ with $n$ samples from $\nu$ for the three example.  For comparison purposes, a sample from $\mu$ is provided, the expectation of which is the limit of $\bb Q_1(\eta)$.}
 \label{tab:ex}
\end{table}

\paragraph{Example 2.}

Let $\mu$ be Gaussian$(0,\sigma^2=2.1)$, $\nu$ be Gaussian$(0,1)$, and $\eta(x) = x^2$.
The estimation of $\int x^2 \dr M$ is natural as the second moment of the unobserved population.
This is an example where both densities are fully supported over $\bb R$.
The density ratio, $\mu(x)/\nu(x) \propto \exp(0.262 x^2)$, is not only unbounded but growing exponentially.
We can see from Figure \ref{fig:gaussian_beta} that near the origin the spread of $M(S_j)$ is low, but far from the origin there is a larger spread and downward bias (in the finite sample).
Due to this bias, the convergence of this example to its expectation is somewhat slower with a $0.8\%$ relative error at 10K samples (Table \ref{tab:ex}).

\paragraph{Example 3.}

In order to see the effect of non-uniform convergence of the LDT we will use a pedagogical construction, the fat Cantor set (the Smith-Volterra-Cantor set).
This set is constructed by the following algorithm: start with $\cl C = [0,1]$; for each $l=1,2,\ldots$ remove the middle $1/4^l$ of the remaining intervals, thereby splitting each interval into two parts.
In simulation, we only perform 5 iterations due to our fine grid.
The remaining set $\cl C$ has $\lambda$ measure of $1/2$ but does not contain any open intervals (it is entirely boundary and has no interior).
Let $\nu$ be uniform$(0,1)$ and $\mu$ be uniform$(\cl C)$.
We can make $\eta(x) = 2 \cdot 1\{x \in \cl C \}$ and so $\int \eta \dr M = 2$.
This example has bounded $\mu,\nu$ over a compact domain, and bounded $\eta$.

The fractal nature of this example causes non-uniform convergence of the LDT because we know that the $M$ measure of a small enough interval around $x$ approaches either $0$ (if $x \notin \cl C$) or $1$ (if $x \in \cl C$).
However, the Fat Cantor set looks from afar as if it does have low and high density regions, and this is manifested in the fact that for $x$ within small intervals that were removed, $M(S_1) | X_1 = x$ is non-zero.
In the subfigure to the right of Figure \ref{fig:fatcantor}, we can see the density ratio is $0$ in small intervals but, because these are surrounded by elements within $\cl C$, the Voronoi cells have large measure, $M(S_i)$.
Due to the fractal nature of the fat Cantor set, for any sample size $n$, this effect will always be manifested at some location at a small enough scale.

Regardless of this non-uniform convergence of the LDT, we observe that $\bb Q_1(\eta)$ converges to its limit, because these regions where the LDT has not yet converged are increasingly small.
We see in Table \ref{tab:ex} that with 10K samples, we achieve a relative error of $0.15\%$.

\begin{figure}
\includegraphics[width=\textwidth]{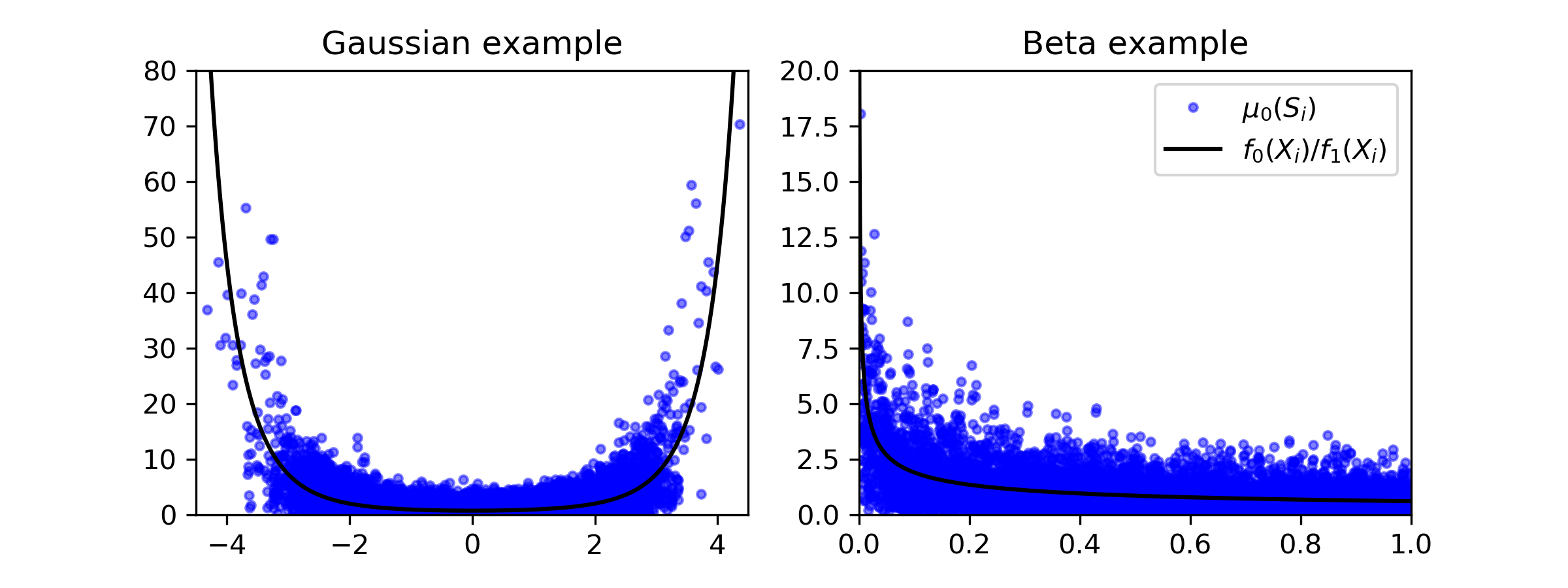}
\caption{1NN interpolation from distributions: left, the Gaussian$(0,1)$ ($\nu$) to Gaussian$(1,2.1)$ ($\mu$); right, the Beta$(1.3,1)$ ($\nu$) to the Beta$(.7,1)$ ($\mu$).
The $\mu$-measure of the Voronoi cell, $S_i$, around sample $X_i\sim \nu$ is plotted and the density ratio $\mu(X_i) / \nu(X_i)$.}
\label{fig:gaussian_beta}
\end{figure}

\begin{figure}
\includegraphics[width=\textwidth]{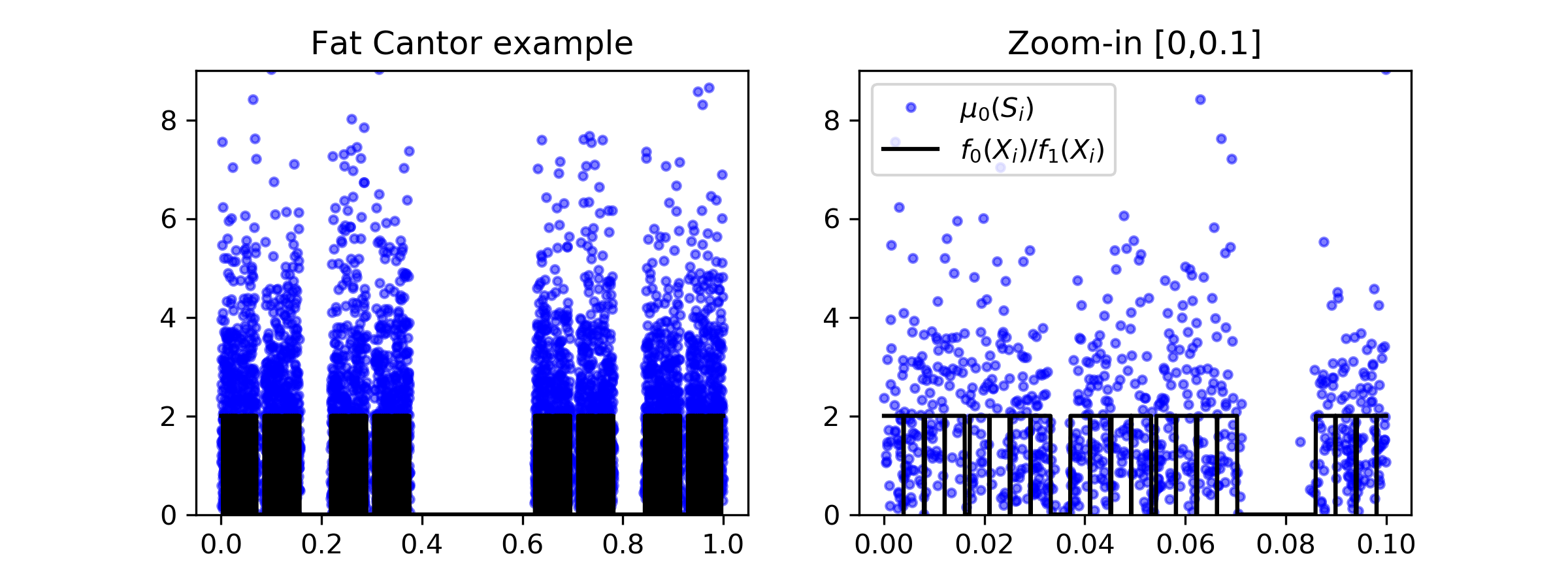}
\caption{1NN interpolation from the uniform$(0,1)$ ($\nu$) to the uniform distribution over the fat Cantor set $\cl C$ ($\mu$).
The $\mu$-measure of the Voronoi cell, $S_i$, around sample $X_i\sim \nu$ is plotted and the density ratio $\mu(X_i) / \nu(X_i) = \bo 1_{\cl C}(X_i)$ (the indicator function over $\cl C$).
The full range (left) and a sub-interval (right) are plotted.}
\label{fig:fatcantor}
\end{figure}

\section{Lemmata}


\begin{lemma}[\cite{biau2015lectures} Lemma 9.1]
\label{lem:stone0}
Suppose that $X,X_1, \ldots, X_n$ are drawn iid from a measure with a density in $\bb R^p$.
Let $g : \bb R^p \rightarrow \bb R$ be a Borel measurable function such that $\bb E|g(X)|^q < \infty$.  Then
\[
\bb E |g (X_{(1)}(X))|^q \le \gamma_p \bb E |g(X)|^q,
\]
where $\gamma_p$ is a universal constant depending on dimension $p$.
\end{lemma}

\begin{lemma}[Stone's Lemma, \cite{stone1977consistent}; \cite{biau2015lectures} Lemma 10.7]
\label{lem:stone1}
Suppose that $X,X_1, \ldots, X_n$ are drawn iid from $N$ (a measure over the Borel $\sigma$-field on $\bb R^p$), and let $X_{(k)}(X)$ denote the $k$NN of $X$ within $X_1,\ldots,X_n$.
Let $v_1, \ldots, v_n$ denote a probability weight vector such that $v_1\ge \ldots \ge v_n$.
Let $g : \bb R^p \rightarrow \bb R$ be a Borel measurable function such that $\bb E|g(X)| < \infty$.  Then
\[
\bb E \sum_{k=1}^n v_k |g (X_{(k)}(X))| \le 2 \gamma_p \bb E |g(X)|,
\]
where $\gamma_p$ is the minimum number of cones of angle $\pi/12$ that cover $\bb R^p$.
\end{lemma}

\begin{lemma}[\cite{biau2015lectures} Lemma 10.2]
\label{lem:stone2}
Suppose that $X,X_1, \ldots, X_n$ are drawn iid from $N$ (a measure over the Borel $\sigma$-field on $\bb R^p$), and let $X_{(k)}(X)$ denote the $k$NN of $X$ within $X_1,\ldots,X_n$.
Let $q \ge 1$, and $g: \bb R^p \rightarrow \bb R$ be a Borel measurable function such that $\bb E|g(X)|^q < \infty$.  Suppose that the following conditions hold:
\begin{enumerate}
\item[(i)] There is a $C$ such that for every Borel measurable $g$, for all $n \ge 1$,
\[
\bb E \left( \sum_{i=1}^n |W_{ni}(X)| |g(X_i)| \right) \le C \bb E |g(X)|.
\]
\item[(ii)] There is a constant $D \ge 1$ such that for all $n \ge 1$,
\[
\bb P \left\{ \sum_{i=1}^n |W_{ni}(X)| \le D \right\} = 1.
\]
\item[(iii)] For all $a > 0$,
\[
\sum_{i=1}^n |W_{ni}(X)| 1 \{ \| X_i - X \| > a \} \rightarrow 0 \quad \textrm{ in probability.}
\]
\end{enumerate}
Then
\[
\bb E \left| \sum_{i=1}^n W_{ni}(X) \left( g(X_i) - g(X) \right) \right|^q \rightarrow 0.
\]
\end{lemma}

\begin{lemma}
\label{lem:two_NN_diff}
Suppose that $X,X_1, \ldots, X_n$ are drawn iid from $N$ (a measure over the Borel $\sigma$-field on $\bb R^p$), and let $X_{(k)}(X)$ denote the $k$NN of $X$ within $X_1,\ldots,X_n$.
Let $q \ge 1$, and $g: \bb R^p \rightarrow \bb R$ be a Borel measurable function such that $\bb E|g(X)|^q < \infty$, then
\[
\bb E \left| g(X_{(1)}(X)) - g(X_{(2)}(X)) \right|^q \rightarrow 0.
\]
\end{lemma}

\begin{proof}
Let $W_{ni}(X) = \frac 12 (1\{ X_{(1)}(X) = X_i \} - 1\{ X_{(2)}(X) = X_i \})$, then letting $v_1 = 1/2, v_2 = 1/2$ we see that condition (i) in Lemma \ref{lem:stone2} holds by Lemma \ref{lem:stone1}.
(ii) holds trivially by selecting $D=1$.
(iii) holds by Lemma 2.2 in \cite{biau2015lectures} which states that for $x \in {\rm supp}(\mu)$, $\| X_{(k)} - x \| \rightarrow 0$ almost surely (for $k/n \rightarrow 0$).
\end{proof}

\section{Proofs of main results}

\begin{proof}[Proof of Lemma \ref{lem:lebesgue_pts}]
Let $\cl A$ be the set of all $x$ such that for some $x' \ne x$, $\Pi(B(x',\| x' - x\|)) = 0$, and call the set of all such balls, $\cl F$.
Let $\cl Z$ be
\[
\cl Z = \bigcup_{B \in \cl F} {\rm int}(B).
\]
Since it is the union of open sets, $\cl Z$ is open, and by the Lindel\"of Covering Theorem, there is a countable subset of $\cl F$, $\cl G$, such that the interiors of the balls cover $\cl Z$.
Thus, by countable subadditivity of measures,
\[
\Pi(\cl Z) \le \Pi\left( \bigcup_{B \in \cl G} {\rm int}(B) \right) \le \sum_{B \in \cl G} \Pi({\rm int}(B)) = 0.
\]
We have that $\cl A \backslash \cl Z$ is $\sigma$-porous which means that there is an $\alpha \in (0,1)$ such that every element $x \in \cl A \backslash \cl Z$ there is an $r_0 > 0$ such that $1 / r_0 \in \bb Z$ where for any $r < r_0$, there exists a $y \in \bb R^p$ with
\[
B(y,\alpha r) \subset B(x,r) \backslash (\cl A \backslash \cl Z).
\]
To see this let $y$ be on the segment between $x'$ and $x$ in the above construction and $\alpha \le 1/2$.
By the Lebesgue differentiation theorem, porous sets have Lebesgue measure $0$ \cite{zajivcek1987porosity}.
Hence, $\Pi(\cl A \backslash \cl Z) = 0$ since $\sigma$-porous sets are countable unions of porous sets, by countable subadditivity, and $\Pi \ll \lambda$.
Let $\cl X^C = \cl A$, and we have that $\Pi(\cl X) \ge 1 - \Pi(\cl A \backslash \cl Z) - \Pi(\cl Z) = 1$.

We will show (2) by supposing its contradiction, that for some $x \in \cl X$ and $\delta > 0$, for every $\gamma >0$ there exists a $x' \ne x$ such that $\| x - x' \| \ge \delta$ and $\Pi(B(x', \|x' - x\|)) \le \gamma$.
This implies that there exists a sequence of points $\{z_l\}_{l=1}^\infty$, such that $\| z_l - x \| = \delta$ and $\Pi(B(z_l,\delta)) \le 1/l^2$.
Define $A_m = \cup_{l=m}^\infty B(z_l,\delta)$ then we have that $\Pi(A_m) \rightarrow 0$ as $m \rightarrow \infty$.
By the Bolzano-Weierstrass theorem there exists an accumulation point of $z_l$, $z'$ with $\|z' - x\| = \delta$ (by continuity of $\|.\|$).
The interior of $B(z',\delta)$ is contained in $A_m$ for all $m$.
By absolute continuity with respect to Lebegue measure, $\Pi({\rm int}(B(z',\delta))) = \Pi(B(z',\delta)) = \Pi(B(z',\|z' - x\|)) >0$ by the fact that $x \in \cl X$.
This contradicts the fact that $\Pi(A_m) \rightarrow 0$.
\end{proof}

\begin{proof}[Proof of Lemma \ref{lem:voronoi}]
Throughout, let $C$ be some constant and $x$ be a Lebesgue point as in Lemma \ref{lem:lebesgue_pts} (for $N $) within ${\rm supp}(M)$.  Let $X_0 \sim M$ and notice that
\[
\bb E[M(S_1) | X_1 = x] = \bb P \{ X_0 \in S_1 | X_1 = x\} = \bb P \left\{ \cap_{i=2}^n \left\{ X_i \notin B(X_0,\| X_0 - x\|) \right\} \right\} = \bb E \left[(1 - Z(x))^{n-1} \right],
\]
where $Z(x) = N (B(X_0,\| X_0 - x\|))$.
By integration by parts,
\[
n \bb E \left[(1 - Z(x))^{n-1} \right] = n \int_0^1 \bb P \left\{ (1 - Z(x))^{n-1} > u \right\} \dr u = n (n-1)\int_0^1 \bb P \left\{ Z(x) \le z \right\} (1 - z)^{n-2} \dr z.
\]

By the Lebesgue differentiation theorem,
\[
\lim_{\delta \rightarrow 0} \sup_{x_0 : \| x - x_0 \| \le \delta} \left| \frac{M(B(x,\|x_0 - x \|))}{\lambda(B(x,\|x_0 - x\|))} - \mu(x)\right| = 0.
\]
Notice that if $\|x_0 - x\| \rightarrow 0$, the sets, $B(x_0,\|x_0 - x \|)$ converges regularly to $x$, in the sense that
\[
B(x_0,\|x_0 - x \|) \subseteq B(x,2 \|x_0 - x \|),
\]
and by the doubling property of $\lambda$,
\[
\lambda(B(x_0,\|x_0 - x \|)) \ge C_{p} \lambda(B(x,2 \|x_0 - x \|)),
\]
where $C_p$ is a constant based on dimension, $p$.  Hence,
\begin{align*}
&\sup_{x_0 : \| x_0 - x \| \le \delta}\left| \frac{N(B(x_0,\|x_0 - x\|))}{\lambda(B(x_0, \|x_0 - x \|))} - \nu(x)\right| \\
&\quad \le \sup_{x_0 : \| x_0 - x \| \le \delta} \frac{\int_{B(x_0,\|x - x_0\|)} \left| \nu(x') - \nu(x) \right| \dr x'}{\lambda(B(x_0,\|x_0 - x \|))} \\
&\quad \le \frac{1}{C_{p} \lambda(B(x,2\delta))} \int_{B(x,2 \delta)} \left| \nu(x') - \nu(x) \right| \dr x' \rightarrow 0,
\end{align*}
as $\delta \rightarrow 0$ again by the LDT.
Because $\lambda(B(x,\|x_0 - x\|)) = \lambda(B(x_0, \|x_0 - x\|))$ we have that,
\begin{equation}
\label{eq:pt_wise}
\lim_{\delta \rightarrow 0} \sup_{x_0 : \| x_0 - x \| \le \delta} \left| \frac{N(B(x_0,\|x_0 - x\|))}{M(B(x, \| x_0 - x\|))} - \frac{\nu(x)}{\mu(x)} \right| = 0.
\end{equation}
For $\gamma > 0$ let $\delta$ be such that for any $x_0$ with $\| x_0 - x \| \le \delta$,
\[
(1 + \gamma)^{-1} \frac{\nu(x)}{\mu(x)} \le \frac{N(B(x_0,\|x_0 - x\|))}{M(B(x, \| x_0 - x\|))} \le (1 - \gamma)^{-1} \frac{\nu(x)}{\mu(x)}.
\]
Let $\eta \in (0,1)$ guaranteed in Lemma \ref{lem:lebesgue_pts} (ii) based on $x,\delta$.
\[
n (n-1) \int_\eta^1 \bb P \{ Z(x) \le z \} (1 - z)^{n-2} \dr z \le n (1 - \eta)^{n-1} \rightarrow 0,
\]
as $n \rightarrow \infty$.

Thus, if we denote $Z_0(x) = M(B(x,\|X_0 - x\|))$,
\begin{align*}
&n (n-1) \int_0^{\eta} \bb P \left\{ \frac{\nu(x)}{\mu(x)} Z_0(x) < (1 - \gamma) z \right\} (1 - z)^{n-2} \dr z \\
&\quad \le n (n-1) \int_0^{\eta} \bb P \{ Z(x) < z \} (1 - z)^{n-2} \dr z \\
&\quad \le n (n-1) \int_0^{\eta} \bb P \left\{ \frac{\nu(x)}{\mu(x)} Z_0(x) < (1 + \gamma) z \right\} (1 - z)^{n-2} \dr z.
\end{align*}
Because $Z_0(x)$ follows a uniform$(0,1)$ distribution then
\begin{align*}
&\lim_{n \rightarrow \infty} n (n-1) \int_0^{\eta} \bb P \left\{ \frac{\nu(x)}{\mu(x)} Z_0(x) < (1 + \gamma) z \right\} (1 - z)^{n-2} \dr z \\
&\quad = \lim_{n \rightarrow \infty} \left( \frac{\mu(x)}{\nu(x)} (1 + \gamma) \right) n (n-1) \int_0^\eta z (1 - z)^{n-2} \dr z \le \frac{\mu(x)}{\nu(x)} (1 + \gamma),\\
\end{align*}
for $n\rightarrow \infty$.
Similarly,
\[
\lim_{n\rightarrow \infty} n (n-1) \int_0^{\eta} \bb P \left\{ \frac{\nu(x)}{\mu(x)} Z_0(x) < (1 - \gamma) z \right\} (1 - z)^{n-2} \dr z \ge \frac{\mu(x)}{\nu(x)} (1 - \gamma).
\]
Hence, by setting $\gamma$ arbitrarily small,
\[
\lim_{n \rightarrow \infty}n (n-1) \int_0^{\eta} \bb P \{ Z(x) < z \} (1 - z)^{n-2} \dr z = \frac{\mu(x)}{\nu(x)}.
\]

In order to establish \eqref{eq:voronoi_sqr}, we will follow a similar procedure.
Let $X_0, X_0' \sim M$ independently.
\begin{align*}
&\bb E[M^2(S_1) | X_1 = x] = \bb P \{ X_0, X_0' \in S_1 | X_1 = x\} \\
&= \bb P \left\{ \cap_{i=2}^n \left\{ X_i \notin B(X_0,\| X_0 - x\|) \cup B(X_0' ,\| X_0' - x\|) \right\} \right\} = \bb E \left[(1 - Z_2(x))^{n-1} \right],
\end{align*}
where $Z_2(x) = N (B(X_0,\| X_0 - x\|) \cup B(X_0',\| X_0' - x\|))$.  Define
\[
\tilde Z_2(x) := \max \{ N (B(X_0,\| X_0 - x\|)), N(B(X_0',\| X_0' - x\|)) \},
\]
then $Z_2(x) \ge \tilde Z_2(x)$.
As before, by integration by parts,
\begin{align*}
&n \bb E \left[(1 - Z_2(x))^{n-1} \right] = n (n-1)\int_0^1 \bb P \left\{ Z_2(x) \le z \right\} (1 - z)^{n-2} \dr z \\
&\quad \le n (n-1)\int_0^1 \bb P \left\{ \tilde Z_2(x) \le z \right\} (1 - z)^{n-2} \dr z.
\end{align*}
By \eqref{eq:pt_wise}, for any $\gamma > 0$ we can select a $\delta$ such that for any $x_0,x_0' \in B(x,\delta)$,
\[
(1 + \gamma)^{-1} \frac{\nu(x)}{\mu(x)} \le \frac{\max\{N(B(x_0,\|x_0 - x\|)),N(B(x_0',\|x_0' - x\|))\}}{\max\{M(B(x, \| x_0 - x\|)),M(B(x, \| x_0' - x\|))\}}.
\]
Let $\eta$ be selected as before,
\[
\int_0^\eta \bb P \left\{ \tilde Z_2(x) \le z \right\} (1 - z)^{n-2} \dr z \le \int_0^\eta \bb P \left\{\frac{\nu(x)}{\mu(x)} Z_3(x) \le (1 + \gamma) z \right\} (1 - z)^{n-2} \dr z
\]
where $Z_3(x) = \max\{M(B(x, \| X_0 - x\|)),M(B(x, \| X_0' - x\|))\}$.
The elements in the maximum are independent uniform$(0,1)$ random variables, and so the maximum has a $\sqrt{U}$ distribution for uniform $U$.
\begin{align*}
&\int_0^\eta \bb P \left\{\frac{\nu(x)}{\mu(x)} Z_3(x) \le (1 + \gamma) z \right\} (1 - z)^{n-2} \dr z \le \int_0^1 \left( \frac{\mu(x)}{\nu(x)}\right)^2 (1 + \gamma)^2 z^2 (1 - z)^{n-2} \dr z \\
&\quad = \left( (1 + \gamma) \frac{\mu(x)}{\nu(x)} \right)^2 \frac{2}{n(n+1)(n-1)}.
\end{align*}
Also, as before
\[
n^2 (n-1) \int_\eta^1 \bb P \left\{ \tilde Z_2(x) \le z \right\} (1 - z)^{n-2} \dr z \rightarrow 0.
\]
Finally, by setting $\gamma$ arbitrarily small
\[
\limsup_{n \rightarrow \infty} n^2 (n-1) \int_0^1 \bb P \left\{ \tilde Z_2(x) \le z \right\} (1 - z)^{n-2} \dr z \le 2 \left( \frac{\mu(x)}{\nu(x)} \right)^2.
\]
\end{proof}

\begin{proof}[Alternative proof of Theorem \ref{thm:bias}]
Consider 
\[
\bb E [\bb Q_1(\eta)] = \bb E \left[\eta(X_1) \cdot n \bb E[ M(S_1) | X_1 ] \right].
\]
We have pointwise convergence by \eqref{eq:voronoi},
\[
\eta(X_1) \cdot n \bb E \left[ M(S_1) | X_1 \right] \rightarrow \frac{\mu(X_1)}{\nu(X_1)} \eta(X_1),
\]
almost everywhere, and the RHS has expectation $\int \eta \mu$.
We can establish dominated convergence by
\begin{align}
&\bb Q_1(\eta) = \int \eta(X_{(1)}(x)) \frac{\mu(x)}{\nu(x)} \nu(x) \dr x \nonumber \\
&\le \left( \int \left( \frac{\mu(x)}{\nu(x)}\right)^{q_0} \nu(x) \dr x \right)^{\frac{1}{q_0}} \cdot \left( \int \eta^{q_1}(X_{(1)}(x)) \nu(x) \dr x \right)^{\frac{1}{q_1}}
\label{eq:holder_Q1}
\end{align}
where $q_0,q_1$ are H\"older conjugates.
By assumption the first term on the RHS is bounded, what remains is to bound the second term.
This can be established using theory developed primarily in \cite{stone1977consistent}.
A direct application of Lemma \ref{lem:stone0} to \eqref{eq:holder_Q1} concludes our proof.
\end{proof}

\begin{proof}[Proof of Lemma \ref{lem:var_bd}]
We will appeal to the Efron-Stein inequality, which states the following:
Let $\bo X' = (X'_1,\ldots,X'_n)$ be an iid copy of $\bo X = (X_1,\ldots,X_n)$ and $\bo X^{(i)} = (X_1,\ldots,X_{i-1},X_i',X_{i+1},\ldots,X_n)$, then for any function $F(\bo X)$
\[
\bb V F(\bo X) \le \frac 12 \sum_{i=1}^n \bb E [F(\bo X) - F(\bo X^{(i)})]^2.
\]
Let $F(\bo X) = \bb Q_1 (\eta)$ and denote $\bb Q^{(i)}$ as the 1NN measure formed from the data, $\bo X^{(i)}$.
Due to exchangeability,
\[
\frac 12 \sum_{i=1}^n \bb E [\bb Q_1 (\eta) - \bb Q^{(i)}(\eta)]^2 = \frac n2 \bb E [\bb Q_1 (\eta) - \bb Q^{(1)}(\eta)]^2.
\]
Let $X^-_{(1)}(x)$ and $\bb Q^-$ denote the 1NN within and the 1NN measure formed from the reduced data $X_2,\ldots,X_n$.
We have that
\[
\bb E [(\bb Q_1 - \bb Q^{(1)})(\eta)]^2 \le \left((\bb E[(\bb Q_1 - \bb Q^-)(\eta)]^2)^{\frac 12} + (\bb E [\bb (\bb Q^- - \bb Q^{(1)})(\eta)]^2)^{\frac 12} \right)^2 = 4 \bb E[(\bb Q_1 - \bb Q^-)(\eta)]^2.
\]
In order for $\eta(X^-_{(1)}(x) )$ to differ from $\eta(X_{(1)}(x) )$ it must be that $X_{(1)}(x) = X_1$ and $X^-_{(1)}(x) = X_{(2)}(x)$.
Thus,
\[
\bb Q_1(\eta) - \bb Q^-(\eta) = \int 1\{ X_{(1)}(x) = X_1 \} \left[ \eta(X_1) - \eta(X_{(2)}(x)) \right] \mu(x) \dr x,
\]
and so,
\begin{align*}
&\bb E[\bb Q_1(\eta) - \bb Q^-(\eta)]^2 \le \bb E \int 1\{ X_{(1)}(x) = X_1 \} \left[ \eta(X_1) - \eta(X_{(2)}(x)) \right]^2 \mu(x) \dr x \\
&= \frac 1n \sum_{j=1}^n \bb E \int 1\{ X_{(1)}(x) = X_j \} \left[ \eta(X_j) - \eta(X_{(2)}(x)) \right]^2 \mu(x) \dr x\\
&= \frac 1n \bb E \int \left[ \eta(X_{(1)}(x)) - \eta(X_{(2)}(x)) \right]^2 \mu(x) \dr x.
\end{align*}
Let $f = \mu / \nu$ be the density ratio.
Considering this term,
\begin{align*}
&\int \left[ \eta(X_{(1)}(x)) - \eta(X_{(2)}(x)) \right]^2 \mu(x) \dr x = \int \left[ \eta(X_{(1)}(x)) - \eta(X_{(2)}(x)) \right]^2 f(x) \nu(x) \dr x \\
&\le \left(\int f(x)^{q_0} \nu(x) \dr x \right)^{\frac{1}{q_0}} \left(\int \left[ \eta(X_{(1)}(x)) - \eta(X_{(2)}(x)) \right]^{2 q_1} \nu(x) \dr x \right)^{\frac{1}{q_1}}.
\end{align*}
\end{proof}

\begin{proof}[Proof of Theorem \ref{thm:noiseless_NNM}]
The random vector $( m \hat M(S_j) )_{j=1}^n$ is multinomial$( m, (M(S_j) )_{j=1}^n)$ conditional on $\cl X_N$.
The MSE
\[
\bb E[ \hat G - G]^2 \le \bb E[ \hat G - \bb Q_1(\eta)]^2 + \bb V [ \bb Q_1(\eta) ].
\]
The second term converges to $0$ by Theorem \ref{thm:1NN_SLLN}.
\[
\bb E[ \hat G - \bb Q_1(\eta)]^2 = \bb E [\bb V[ \hat G | \cl X_N]].
\]
The conditional variance is
\begin{eqnarray*}
&\bb V[ m \hat G | \cl X_N] = \sum_{j,j' = 1}^n {\rm Cov} (m \hat M(S_j), m \hat M(S_{j'}) | \cl X_N) \eta(X_j) \eta (X_{j'})\\
&= \sum_{j=1}^n m M(S_j) \eta^2(X_j) - \left( \sum_{j,j' = 1}^n M(S_j) M(S_{j'}) \eta(X_j) \eta (X_{j'}) \right)\\
&\le \sum_{j=1}^n m M(S_j) \eta^2(X_j) = m \bb Q_1 (\eta^2).
\end{eqnarray*}
Hence, under Assumptions \ref{as:measure_1}, \ref{as:measure_2},
\[
\bb E [\bb V[ \hat G | \cl X_N]] \le \frac 1m \bb E [\bb Q_1 (\eta^2)] \rightarrow 0.
\]
by Theorem \ref{thm:bias}.
(Notice that Theorem \ref{thm:bias} only requires the $q_1$ moment bound of the test function, which is satisfied for $\eta^2$ by Assumption \ref{as:measure_2}.)
\end{proof}

\begin{proof}[Proof of Theorem \ref{thm:noisy_NNM}]
Define $\tilde G = \sum_j \hat M(S_j) \cdot \eta(X_j)$ ($\hat G$ in the noiseless setting).
Let $\cl X = \cl X_N \cup \cl X_M$ be all of the covariates,
\[
\bb E [\hat G - G]^2 \le \bb E[\bb E [(\hat G - \tilde G)^2 | \cl X]] + \bb E[ \tilde G - G]^2.
\]
The last term converges to $0$ by Theorem \ref{thm:noiseless_NNM}.
The inner term is dominated because
\[
\bb E [(\hat G - \tilde G)^2 | \cl X] \le V \sum_j \hat M^2(S_j) \le V,
\]
because $\sum_j \hat M^2(S_j) \le \sum_j \hat M(S_j) = 1$.
Consider 
\[
\bb E \left[\sum_j \hat M^2(S_j)\right] = n \bb E[\hat M^2(S_1)] = n \bb E \left[ \bb E [\hat M^2(S_1) | \cl X_N] \right].
\]
Because $m \hat M(S)$ is binomial$(m,M(S))$ for fixed $S$ we have that,
\[
\bb E [\hat M^2(S_1) | \cl X_N] = \frac{1}{m^2} \cdot m (m-1) M^2(S_1) + \frac 1m M(S_1).
\]
Since $n \bb E [M(S_1)] = 1$ we have that $ \frac nm \bb E [M(S_1)] \rightarrow 0$ if $m \rightarrow \infty$.
We have by Lemma \ref{lem:voronoi} and dominated convergence,
\[
\lim\sup_{n \rightarrow \infty} n^2 \bb E \left[ \bb E [M^2(S_1) | X_1] \right] \le 2 \int \left(\frac{\mu(x)}{\nu(x)}\right)^2 \nu(x) \dr x.
\]
Hence,
\[
n \bb E[ \hat M^2(S_1)] \rightarrow 0.
\]
\end{proof}

\end{document}